\documentclass[12pt]{amsart}
\usepackage{amsmath,amssymb,amsfonts,amsthm,xspace}
\usepackage{graphicx}
\usepackage{geometry}
\usepackage{epstopdf}
\usepackage[linktocpage=true]{hyperref}
\usepackage[usenames,dvipsnames]{color}

\newtheorem{theorem}{Theorem}
\newtheorem{prop}[theorem]{Proposition}
\newtheorem{lemma}[theorem]{Lemma}

\newtheorem{corollary}[theorem]{Corollary}

\newcommand{\be}{\begin{equation}}
\newcommand{\ee}{\end{equation}}
\newcommand{\Z}{{\mathbb Z}}
\newcommand{\R}{{\mathbb R}}

\newcommand{\F}{\mathbb F}
\newcommand{\M}{{\mathcal M}}
\newcommand{\C}{{\mathbb C}}

\newcommand{\G}{{\mathcal G}}
\newcommand{\U}{\mathbb U}

\newcommand{\hol}{\text{hol}}

\newcommand{\old}[1]{}

\newcommand{\A}{\mathbb A}

\renewcommand{\S}{\mathcal S}
\renewcommand{\P}{\mathrm{Pr}}

\renewcommand{\H}{{\mathbb H}}

\begin{document}

\pagestyle{plain}

\title{The singularity probability of random diagonally-dominant Hermitian matrices}
\author{Adrien Kassel}
\address{Adrien Kassel\\ ETH Z\"urich\\ Departement Mathematik\\ R\"amistrasse 101\\ 8092 Z\"urich, Switzerland}
\email{adrien.kassel@math.ethz.ch}

\begin{abstract} 
In this note we describe the singular locus of diagonally-dominant Hermitian matrices with nonnegative diagonal entries over the reals, the complex numbers, and the quaternions.
This yields explicit expressions for the probability that such matrices, chosen at random, are singular. For instance, 
in the case of the identity $n \times n$ matrix perturbed by a symmetric, zero-diagonal, $\{\pm 1/(n-1)\}$-Bernoulli matrix, this probability 
turns out to be equal to $2^{-(n-1)(n-2)/2}$. As a corollary, we find that the probability for a $n\times n$ symmetric $\{\pm 1\}$-Bernoulli matrix to have eigenvalue~$n$ is $2^{-(n^2-n+2)/2}$.
\end{abstract}

\maketitle

\section{Introduction}

It is a famous open problem in non-asymptotic random matrix theory to compute the probability $P_n$ that a Bernoulli matrix of size $n\times n$ is singular. The asymptotic value of $P_n$ is conjectured to be $n^2 2^{-(n-1)}(1+o(1))$ as $n\to\infty$, which coincides with the asymptotic value of the probability $2 {n \choose 2} 2^{-(n-1)}$ that two rows or two columns are collinear. See~\cite{Komlos,KKS,TaoVu,BVW} for the history of this problem and major breakthroughs. The same question was addressed in~\cite{CTV,Nguyen,Costello,Vershynin} in the case where the Bernoulli matrix is conditioned to be symmetric.

In this note, we consider a problem which has a similar flavor but turns out to be much easier. Let $B_n$ be a random $n\times n$ symmetric matrix with zero diagonal and strict upper-diagonal entries equal to independent $\{\pm 1\}$-Bernoulli random variables. Set $M_n=I_n+B_n/(n-1)$, where $I_n$ is the identity matrix. A special case of our Theorem~\ref{bernou} is that for all $n\ge 1$, the probability that $M_n$ is singular is
\begin{equation}\label{qn}
\P\left(M_n\,\text{is singular}\right)=2^{-(n-1)(n-2)/2}\,.
\end{equation}
Note that this quantity is also the probability that $n-1$ is an eigenvalue of $B_n$.

Consider now a random $n\times n$ symmetric matrix $Q_n$ with upper-diagonal entries equal to independent $\{\pm 1\}$-Bernoulli random variables. The best current bound~\cite{Vershynin} on the asymptotic probability that $Q_n$ has eigenvalue $0$ is~$2\exp(-n^c)$, where~$c>0$ is a constant. A special case of our Corollary~\ref{coro} is that for all $n\ge 1$, the probability that $Q_n$ has eigenvalue~$n$ (the maximal possible value for its eigenvalues) is
\begin{equation}\label{eigenvalue}
\Pr\left(Q_n\,\text{has eigenvalue}\,\,n\right)=2^{-(n^2-n+2)/2}\,.
\end{equation}

\section{Determinant expansion}

All matrices that we will consider take values in $\F$, which is either the field~$\R$ of real numbers, the field~$\C$ of complex numbers, or the division algebra~$\H$ of quaternions. We denote by~$M_n(\F)$ the space of $n\times n$ matrices over $\F$ and by $S^+_n(\F)$ the space of self-dual matrices (that is symmetric, Hermitian, and quaternion-Hermitian when $\F=\R, \C$, and $\H$, respectively) with nonnegative diagonal entries. (For background on quaternion-Hermitian matrices, see e.g. \cite[Chapter~5]{Mehta}.) Note that the diagonal entries of a self-dual matrix are necessarily real, as are its eigenvalues.

A matrix $M\in M_n(\F)$ is said to be \emph{diagonally-dominant} if for all $1\le i\le n$, we have
\begin{equation}|M_{i,i}|\ge \sum_{j\ne i}|M_{i,j}|\,.\label{def}\end{equation}
It is well known that if all inequalities in~\eqref{def} are strict, then~$M$ is nonsingular.

If we furthermore assume $M\in S_n^+(\F)$, then $M$ is positive semidefinite. Indeed, one can check that for all $X\in\F^n$, 
\[XM\overline{X}^t=\sum_{i=1}^n \left(M_{i,i}-\sum_{j\ne i}|M_{i,j}|\right)|X_i|^2+\sum_{\{i,j\}, i\ne j}|M_{i,j}|\left\vert X_j+\frac{M_{i,j}}{|M_{i,j}|}X_i\right\vert^2\,,\]
where the second sum is over unordered pairs of distinct indices such that $M_{i,j}\ne 0$.  (This formula should remind the reader of the Dirichlet energy of a function $X$ defined on the vertices of a graph and this interpretation will appear clearer by the end of this section.)
In particular, the spectrum of~$M$ is real and nonnegative. The matrix~$M$ is thus singular whenever its smallest eigenvalue takes the extremal value~$0$.

Recall that a matrix $M$ is \emph{block-diagonal} if there are two disjoint subsets of indices $S,S'$ such that the blocks $(M_{i,j})_{i\in S,j\in S'}$ and $(M_{i,j})_{i\in S',j\in S}$ are zero. If $M$ is not block-diagonal (which is the case generically), it then suffices that one of the inequalities in~\eqref{def} be strict in order for $M$ to be nonsingular. A proof of this fact in the case $\F\subset\C$ can be found in~\cite[Theorem~6.2.27]{HJ}. We give another proof, which has the advantage to work also over the quaternions.

\begin{prop}\label{dominant}
Let $M\in S^+_n(\F)$ be diagonally-dominant. Suppose that $M$ is not block-diagonal. If for some $i\in \{1,\ldots,n\}$, we have
\[
M_{i,i}>\sum_{j\ne i}|M_{i,j}|\,,
\]
then $M$ is nonsingular.
\end{prop}

The proof of Proposition~\ref{dominant} is based on the following construction, introduced in~\cite{Kas}.

Let us first recall some definitions, which may also be found e.g. in~\cite{Ken11}. Let~$\G$ be a finite connected graph with edge-set~$E$ and vertex-set~$V$. Let $\U$ be the group of modulus~$1$ elements of $\F$. A unitary \emph{connection}~$\varphi$ is a collection of elements $\varphi_{v,v'}\in \U$, one for each oriented edge $vv'$, satisfying~$\varphi_{v,v'}=\varphi_{v',v}^{-1}$. The element $\varphi_{v,v'}$ is called the \emph{parallel transport} of~$\varphi$ from $v$ to $v'$.
The parallel transport of~$\varphi$ along a path in~$\G$ is the product of the $\varphi_{v,v'}$s along that path (for the simplicity of exposition, we choose the product to be taken in the order of the path, and not backwards as is customary). The parallel transport of~$\varphi$ along a closed path~$\gamma$ is called the \emph{holonomy} of~$\varphi$ around~$\gamma$. Changing the base point of~$\gamma$ conjugates the holonomy around~$\gamma$ by an element of~$\U$. Changing the orientation of~$\gamma$ inverts it.

Let us associate to a matrix $M\in S_n^+(\F)$ a weighted undirected graph $\G$ on the vertex set $V=\{1,\ldots,n\}$ with unitary connection. For any $i\ne j$, the graph has edge~$ij$ if $M_{i,j}\ne 0$, its weight is $w(ij)=|M_{i,j}|$, and the unitary parallel transport from~$i$ to~$j$ is~$\varphi_{i,j}=-M_{i,j}/|M_{i,j}|$. Note that the assumption that $M$ is not block-diagonal implies that $\G$ is connected. The (massless) \emph{Laplacian} associated to this data is the matrix~$\Delta$ with entries 
\[\Delta_{i,j}=-w(ij)\varphi_{i,j}=M_{i,j}\] for $i\ne j$, and 
\[\Delta_{i,i}=\sum_{j\ne i}w(ij)=\sum_{j\ne i}|M_{i,j}|\] 
for all $i$.

Given a subset of vertices $S\subset V$, an essential \emph{cycle-rooted spanning forest} (CRSF) is a spanning subgraph of~$\G$, each of whose components not intersecting~$S$ is a unicycle (that is, has as many edges as vertices) and such that all other components are trees rooted at vertices of~$S$ (one per vertex of~$S$). Let $\det M_S$ be the principal minor of $M$, where the rows and columns indexed by~$S$ have been removed.
In the quaternion case, we choose Moore's quaternion determinant~\cite{Mo}, which we still denote by~$\det$. It is well known that a quaternion-Hermitian matrix is singular if and only if its determinant is zero, see e.g.~\cite[Chapter~5]{Mehta}.

A matrix $M\in S_n^+(\F)$ such that for all rows $i\in V$, we have $M_{i,i}=\sum_{j\neq i}|M_{i,j}|$ is the Laplacian $\Delta$ associated to the weights and unitary connection defined above (this is obvious from the definition of this Laplacian).  By~\cite[Theorem~9]{Ken11} (and by~\cite[Theorem~1]{Fo} for the case $\F\subset\C$), we have in that case
\begin{equation}\label{detdelta}
\det M_S=\sum_{\Gamma} \prod_{e\in\Gamma} w(e)\prod_{\gamma\subset\Gamma}\left(2-2\,\Re[\hol(\gamma)]\right)\,,
\end{equation}
where the sum is over all essential CRSFs $\Gamma$ of $\G$, the first product is over the edges~$e$ of~$\Gamma$, the second product is over the (simple) cycles~$\gamma$ of~$\Gamma$, and $\hol(\gamma)$ is the holonomy of $\varphi$ around the cycle~$\gamma$. Note that the real part~$\Re[\hol(\gamma)]$ of the holonomy is independent of the choice of orientation and base point of~$\gamma$.

Let us now prove Proposition~\ref{dominant}.
\begin{proof}[Proof of Proposition~\ref{dominant}]

The matrix $M$ is a massive Laplacian associated to $\varphi$. That is, we can write $M=D+\Delta$, where $\Delta$ is the (massless) Laplacian corresponding to the above data and $D$ is a diagonal matrix with the nonnegative entries 
\[D_{i,i}=M_{i,i}-\sum_{j\neq i}|M_{i,j}|\,.\] 
It is easy to verify (in the quaternion case as well, since the $D_{i,i}$s are real, and thus commute with all other entries) that
\begin{equation}\label{expansion}
\det M =\det\bigl(D+\Delta\bigr)=\sum_{S\subset V}\det(\Delta_{S})\prod_{i\notin S}D_{i,i}\,.
\end{equation}
Note that the right-hand side of~\eqref{detdelta} is strictly positive when $S$ is nonempty. Indeed, all terms in the sum are nonnegative by unitarity of the connection, and furthermore, there is a nonzero term coming from the existence (since the graph is connected) of a spanning tree rooted at $S$ (corresponding to an essential CRSF without cycles). 
When $D_{i,i}\neq 0$ for some~$i\in V$, it therefore follows from~\eqref{expansion} that $\det M$ is nonzero.
\end{proof}

In order to investigate the singularity of (generic) diagonally-dominant Hermitian matrices, Proposition~\ref{dominant} shows that we may restrict our attention to those where equality holds in~\eqref{def} for all $i\in\{1,\ldots,n\}$.

\section{Topological characterization}

Let $\G$ be a finite connected graph with edge-set~$E$ and vertex-set~$V$. Let $\A$ be a subgroup of $\U$. A connection with values in $\A$ is a unitary connection such that $\varphi_{v,v'}\in\A$ for all $vv'\in E$. A \emph{gauge transformation} with values in $\A$ is the action of $\psi=\left(\psi_v\right)\in \A^V$ on connections~$\varphi$ by $\left(\psi\cdot\varphi\right)_{v,v'}=\psi_{v}\varphi_{v,v'}\psi_{v'}^{-1}$ (this definition, although not customary, is consistent with our earlier choice of product order in the definition of the holonomy).
Note that a gauge transformation modifies the holonomy of a connection along a closed path by conjugating it with an element of~$\A$.

Let $\M$ denote the set of unitary connections of~$\G$ with values in $\A$ modulo gauge transformation (this equivalence relation is called~\emph{gauge equivalence}). We define the \emph{trivial connection} to be the constant connection equal to $1$ on all edges, and we define the \emph{trivial} element of $\M$ to be the gauge-equivalence class of the trivial connection.

When $\A$ is abelian, $\M$ has a cohomological interpretation. It is isomorphic to the cohomology group $H^1\left(\G,\A\right)$, which, by the universal coefficient theorem (see e.g.~\cite[Chap. VI, Theorem~3.3.a]{Cartan}), is isomorphic to
$\mathrm{Hom}\left(H_1\left(\G,\Z\right),\A\right)$. The homology group $H_1\left(\G,\Z\right)$ being isomorphic to $\Z^r$, with $r=|E|-|V|+1$, we have that $\M$ is isomorphic to~$\A^r$. (Further note that in our case, $\A$, being an abelian subgroup of $\U$, is necessarily cyclic.)

Let $M\in S_n^+(\F)$ be such that all inequalities in~\eqref{def} are equalities and assume $M$ is not block-diagonal. Let~$\G$ be the corresponding connected graph and let $\A$ be the subgroup of~$\U$ generated by the $\varphi_{i,j}$s. Let $[\varphi]\in\M$ be the gauge-equivalence class of the connection~$\varphi$. 

\begin{lemma}\label{main}
The matrix $M$ is singular if and only if $[\varphi]$ is trivial.
\end{lemma}
\begin{proof}
The determinant of~$M$ is given by~\eqref{detdelta} and it is zero if and only if all terms in the sum vanish (since they are nonnegative). Note that any cycle in~$\G$ is contained in a CRSF, since there exists a spanning tree rooted on that cycle (by connectedness of~$\G$). Hence, by~\eqref{detdelta}, the determinant of $M$ is zero if and only if for all cycle $\gamma$ in~$\G$, the real part of the holonomy of $\varphi$ around $\gamma$ is $1$, which implies (by unitarity of~$\varphi$) that the holonomy itself is $1$. 

If $[\varphi]$ is trivial, then the holonomy of~$\varphi$ around any cycle is $1$, which implies that $\det M_S=0$ in view of~\eqref{detdelta}. Let us prove the converse implication. It is well known that given any spanning tree $t$ of~$\G$, any connection~$\varphi$ can be gauge-transformed (by means of a gauge transformation with values in~$\A$) to a connection~$\widetilde{\varphi}$ which is equal to~$1$ on all edges of~$t$. If the holonomies of~$\varphi$ along cycles are supposed to be~$1$, then so is the case for~$\widetilde{\varphi}$. This implies that the parallel transport of $\widetilde{\varphi}$ along all edges outside~$t$ is also~$1$ and therefore $\widetilde{\varphi}$ is the trivial connection. Hence $[\varphi]$ is trivial.
\end{proof}

\section{Main result}

Any probability distribution on $S_n^+(\F)$ induces a probability distribution on $\M$. By virtue of Lemma~\ref{main}, the probability that a random $M\in S_n^+(\F)$ is singular is the probability that the corresponding random element $[\varphi]$ of $\M$ is trivial. 

As we have seen, the singularity of $M\in\S_n^+(\F)$ depends only on the `projective' variables~$\varphi_{i,j}=-M_{i,j}/|M_{i,j}|$ (and in fact, only on the gauge-equivalence class of~$\varphi$). For simplicity, we thus state our main result in the case where all off-diagonal entries have the same modulus. Relaxing this assumption leaves the conclusion unchanged.

Let $\mathbb{A}$ be a finite subgroup of $\U$ and let $q$ be its cardinality. Note that~$\A$ is not assumed to be abelian. Let $B_n$ be a random self-dual matrix with zero diagonal and strict upper-diagonal entries given by independent uniform random variables on $\mathbb{A}$. Let $M_n=I_n+B_n/(n-1)$.

We now state our main result which gives~\eqref{qn} in the special case $\A=\{-1,1\}$.

\begin{theorem}\label{bernou}
For all $n\ge 1$, the probability that $M_n$ is singular is
$$\P\left(M_n\,\text{is singular}\right)=q^{-(n-1)(n-2)/2}\,.$$
\end{theorem}
\begin{proof}
It is easy to show that the trivial element of $\M$ is the gauge-equivalence class of exactly $q^{|V|-1}$ distinct connections with values in~$\A$. Indeed, consider the map
$f:\A^V\to\A^E, (u)_{i\in V}\mapsto (u_i u_j^{-1})_{ij\in E}$ corresponding to the action of gauge transformations on the trivial connection. Since $f$ is invariant under right-multiplication by a common element of $\A$ and since the graph is connected, we see that $f(u)=f(v)$ if and only if there exists $a\in\A$ such that for all $i\in V$, $u_i=v_i a$. Hence there are exactly $q^{|V|-1}$ distinct connections gauge-equivalent to the trivial connection.

Since the total number of connections with values in~$\A$ is $q^{|E|}$, Lemma~\ref{main} implies that
\begin{equation}\label{graph}
\P\left(M_n\,\text{is singular}\right)=\frac{q^{|V|-1}}{q^{|E|}}=q^{-(|E|-|V|+1)}\,,
\end{equation} 
where $|E|-|V|+1=(n-1)(n-2)/2$, since the graph is the complete graph on $|V|=n$ vertices, which has $|E|={n \choose 2}$ edges. 
\end{proof}

\section{A corollary}

Suppose $-1\in\A$ and let $Q_n$ be a random self-dual matrix with strict upper-diagonal entries given by independent uniform random variables on~$\A$ and diagonal entries given by independent $\{\pm 1\}$-Bernoulli random variables.

We now state a corollary of Theorem~\ref{bernou} which gives~\eqref{eigenvalue} in the special case $\A=\{-1,1\}$.

\begin{corollary}\label{coro}For all $n\ge 1$, the probability that $Q_n$ has eigenvalue $n$ is
\[\Pr\left(Q_n\,\text{has eigenvalue}\,\,n\right)=2^{-n}q^{-(n-1)(n-2)/2}\,.\]
\end{corollary}
\begin{proof}
The matrix $\widetilde{M_n}=nI_n-Q_n$ is in $S_n^+(\F)$. The matrix $Q_n$ has eigenvalue $n$ if and only if $\widetilde{M_n}$ is singular.
By Proposition~\ref{dominant}, $\widetilde{M_n}$ is nonsingular if one of the inequalities in~\eqref{def} is strict. In order for $\widetilde{M_n}$ to be singular, the diagonal entries of~$Q_n$ must therefore be all equal to $1$. This event occurs with probability $2^{-n}$.
Conditional on this event, the probability that $\widetilde{M_n}$ is singular (which is the probability that $M_n=\widetilde{M_n}/(n-1)$ is singular) is~$q^{-(n-1)(n-2)/2}$, using Theorem~\ref{bernou} (since the off-diagonal entries of~$\widetilde{M_n}$ are independent of its diagonal entries). The result follows. 
\end{proof}

\section{Concluding remarks}

Note that when the matrix $M_n=I_n+B_n/(n-1)$ is singular, its kernel is one-dimensional. This follows from the fact that the elements of the kernel are harmonic sections of the vector bundle~$\F^V$ over~$\G$ and that the graph is connected (this was shown e.g. in~\cite{Kas} and we leave it as an exercise to the reader). 

The expected value of $\det M_n$ is obviously~$1$ and it is also easy to see that the expected value of the smallest eigenvalue of~$M_n$ is less than or equal to $1$. It would be interesting to compute the actual distributions of~$\det M_n$ and of the smallest eigenvalue. Note that the \emph{spectrum} of~$M_n$, seen as a function of the corresponding connection~$\varphi$, is invariant under gauge transformation (since the gauge transformation of $\varphi$ corresponds to the conjugation of $M_n$ by a diagonal matrix with entries in~$\A$) and may therefore be considered as a function on~$\M$.

Using our method, Theorem~\ref{bernou} and Corollary~\ref{coro} can be extended to the case of more general distributions on the entries of~$M_n$ (in particular, the off-diagonal entries need not all have the same modulus). Furthermore, it is also possible to allow for the entries to take zero values (and hence the matrix to be sparse), since this simply corresponds to erasing edges in the corresponding graph, a case which can also be handled using Lemma~\ref{main}.

A similar approach can be used to study non-Hermitian matrices~$M$, since assuming that each diagonal entry is equal to the sum of the moduli of the other entries on that row, the matrix $M$ becomes the Laplacian of a directed graph with connection. However, the fact that the graph is now directed complicates the study and in particular Lemma~\ref{main} is no longer valid. Indeed, whatever the value of the connection, a directed graph may contain no oriented cycle. This implies that the determinant of~$M$ is zero by the analog of~\eqref{detdelta} for directed graphs, where the right-hand side is then a sum over oriented CRSFs, see~\cite{Fo,Ken11}. Hence, the fact that~$M$ is singular does not only depend on the connection~$\varphi$ but also on the structure of the corresponding directed graph.

Let us conclude this note by mentioning that the main tool of this paper, namely graphs with connections, has been known under the name of \emph{gain graphs} in the combinatorics literature since at least~\cite{Zaslavsky}, and seemingly independently of the physics/geometry literature from where it stems. A special case of a gain graph is a \emph{signed graph} which we now define using our terminology. A signed graph is simply an unweighted graph with real unitary connection: it gives a positive sign ($+1$ parallel transport) or negative sign ($-1$ parallel transport) to any edge, regardless of orientation. A signed graph is said to be \emph{balanced} if the gauge-equivalence class of the connection is trivial in~$\M$. Harary~\cite[Theorem~3]{Harary} gave a combinatorial interpretation for a graph to be balanced: a graph is balanced if and only if the vertices may be partitioned in two disjoint subsets, such that any two neighboring vertices are connected by a positive edge if and only if they belong to the same subset. 

From this point of view, let us give a final application of our method. Consider \emph{bond percolation} with parameter $p=1/2$ on a connected graph~$\G$ (that is, declare edges \emph{open} or \emph{closed} with probability $1/2$, independently). We may view this as a random signed graph where open edges have a negative sign and closed edges a positive sign. We say that the percolation configuration contains \emph{fjords} if there are at least two adjacent vertices of the same \emph{cluster} (connected by a path of closed edges) separated by an open edge. The probability that there are no fjords is the probability for the random signed graph to be balanced. By Lemma~\ref{main}, this is the same as the probability for the corresponding Laplacian to be singular, which by~\eqref{graph} is~$2^{-|E+|V|-1}$. For the complete graph $\G=K_n$, this probability is thus~$2^{-(n-1)(n-2)/2}$.

\section*{Acknowledgments}

We thank Rajat Subhra Hazra for bringing~\cite{TaoVu} to our attention and for pointing out to us~\cite[Theorem~6.2.27]{HJ}. We also thank him, as well as Richard Kenyon, Wendelin Werner, and David Wilson for feedback on a preliminary version of this note.

\end{document}